\documentclass[12pt,letterpaper]{amsart}
\usepackage{sty_equiv_GW}

\begin{document}
\title{A Quadratically Enriched Pushforward for Deligne--Mumford Stacks}

\author[Bethea]{Candace Bethea}\author[Wickelgren]{Kirsten Wickelgren}

\begin{abstract}
We define a quadratically enriched pushforward for proper morphisms between smooth, proper Deligne--Mumford stacks over an arbitrary field using Grothendieck--Serre duality. We apply this pushforward to define quadratically enriched Euler classes and Euler numbers of relatively oriented vector bundles on smooth, proper Deligne--Mumford stacks. Finally, we define quadratically enriched fundamental classes for smooth, proper schemes over an arbitrary field, and we use these to propose a definition of quadratically enriched quantum $K$-invariants when the moduli space of stable maps is smooth, tame, and relatively oriented.
\end{abstract}
\maketitle 

\begingroup
\setlength{\parskip}{3pt}
\renewcommand{\baselinestretch}{0.25}\normalsize
\tableofcontents
\renewcommand{\baselinestretch}{1.0}\normalsize
\endgroup

\section{Introduction}\label{section:intro}

Quadratically enriched enumerative geometry has been successful in extending classical enumerative results that were once only attainable over $\C$ and $\R$ to arbitrary fields by producing enriched counts, valued in the Grothendieck--Witt ring of quadratic forms over a general field $k$. Initially spearheaded by Hoyois, Kass--Wickelgren, and Levine \cite{Hoyois-lefschetz, KW-local-degree, levineaspects}, examples of quadratically enriched  counts of classical enumerative problems include \cite{KW-local-degree, 27lines, Hoyois-lefschetz, SW-4-lines, quadratic-tropical, quadratic-twisted-cubics,KLSW-rational-curves}. See \cite{bb-history} for an in-depth expository review of results in enriched enumerative geometry more generally, with an emphasis on quadratically enriched developments. 

The homotopical invariants used in quadratically enriched enumerative geometry up to this point, including various definitions of the $\mathbb{A}^1$-degree, the enriched Euler characteristic, and the enriched Euler number \cite{KW-local-degree, KLSW-rational-curves, BW-Euler, levine-raskit}, have primarily been defined for schemes using six functor formalisms such as that due to Ayoub, Hoyois, and others \cite{ayoub-1, Hoyois6ff}. Six functor formalisms for certain sufficiently nice stacks exist, such as those due to Hoyois and Khan--Ravi \cite{Hoyois6ff, khan-ravi}, and have recently been used in equivariant motivic enriched enumerative results, see \cite{bethea-ravi}. 

This work defines a quadratically enriched pushforward for proper maps between smooth, proper Deligne--Mumford stacks with projective coarse moduli spaces over any field $k$, which is used to define a quadratically enriched Euler number of a vector bundle on a smooth, proper Deligne--Mumford stack. We do this by defining a naive bilinear form pushforward and composing with a canonical map to the Grothendieck--Witt ring of a field. The key construction of the quadratically enriched pushforward is as follows: 

\begin{theorem}\label{thm:intro_pushforward}
    Let $f: X \to Y$ be a morphism between  Deligne--Mumford stacks $X$ and $Y$, which are proper and smooth over $\Spec k$ of dimension $d_X$ and $d_Y$ respectively. Assume $X$ and $Y$ have projective coarse moduli spaces. Then for any invertible sheaf $\calL$ on $Y$ there is a pushforward 
\[
f_*: \Bl^{\naive}(X,f^*\calL \otimes \omega_f [d_f+n]) \to \Bl^{\naive}(Y,\calL[n])
\] defined by sending \[
\beta: C_* \otimes C_* \to f^*\calL \otimes \omega_f [d_f+n]
\] in $\Bl^{\naive}(X,\calL \otimes \omega_f [d_f+n])$ to 
\[
 Rf_* C_* \otimes  Rf_* C_* \to  Rf_*(C_* \otimes C_*) \xrightarrow{Rf_*\beta} Rf_* (f^*\calL \otimes \omega_f [n+d_f]) \simeq \calL \otimes Rf_*(\omega_f [n+d_f]) \xrightarrow{\Tr_f} \calL[n].
\] 
\end{theorem}

See Theorem \ref{thm:pushforward} for the proof. This is a pushforward between naive bilinear form rings rather than Grothendieck--Witt rings, defined as follows. For an invertible sheaf $\calL$ on $X$ and integer $n$,  $\Bl^{\naive}(X,\calL[n])$ denotes the monoid of isomorphism classes of symmetric bilinear forms
\[
\beta: C_* \otimes C_* \to \calL[n] 
\]  under $\oplus$, with $C_*$ in $D_c^\flat(X)$, which are unimodular. See Definition \ref{bl_naive}. 

Constructing a pushforward in naive bilinear form groups 
allows us to use Grothendieck--Serre duality for Noetherian, separated Deligne--Mumford stacks that are finite type over $k$ \cite{lev,nironi}. This is opposed to using a Gysin pushforward belonging to a more sophisticated, fully-fledged six functor formalism for Deligne--Mumford stacks, see \cite{khan-ravi}, which typically requires a representability hypothesis that we wish to avoid for the specific purposes of this paper. 

While the pushforward in Theorem \ref{thm:intro_pushforward} does not fit into a six functor formalism, it is of use in quadratically enriched enumerative geometry. Many quadratically enriched results of interest are valued in $\GW(k)$, and there is a canonical map $\theta\colon \Bl^{\naive}(k, \calO_k)\to \GW(k)$. Thus we may pushforward to $\Spec k$ in the naive bilinear form theory to obtain a naive enriched count, then apply $\theta$ to obtain a quadratically enriched enumerative result in $\GW(k)$. See Definition \ref{df:BlnaivektoGW} for the definition of the canonical map $\theta\colon \Bl^{\naive}(k, \calO_k)\to \GW(k)$.

The pushforward constructed in Theorem \ref{thm:intro_pushforward} allows us to define a quadratically enriched Euler class and Euler number for any relatively oriented vector bundle on a smooth, proper Deligne--Mumford stack over $k$. See Definitions \ref{def:naive_euler} and \ref{def:euler_number} respectively. These Euler numbers also define a Witt invariant in the sense of Serre \cite{GMS}, giving  quadratic enrichments of the Euler characteristics of $\MgnbarXb$ and $\Mgnbar$ when they are smooth, tame, and have projective coarse moduli spaces for a given smooth, proper $X$ over $k$ and $\beta \in H_2(X,\Z)$. Finally, given a smooth, projective, oriented scheme $X$ over $k$ and closed subscheme $Z$, we define quadratically enriched fundamental class of $Z$. Using this and the pushforward, we propose a definition for quadratically enriched quantum $K$-invariants under orientation hypotheses, building on the quantum $K$-invariants of Y.P. Lee \cite[Section 4.2]{lee-quantum-invariants}. See Definition \ref{defn:smooth_quantum_K}. This is certainly a substantial assumption, as orienting $\MgnbarXb$ is a considerable feat even when we do not work over arbitrary fields \cite{solomon-thesis}.

\subsection{Acknowledgments} We gratefully acknowledge helpful discussions with Ruijie (Jared) Zhang (especially for Section ~\ref{Subsection:FundClassSubLNew}), Jack Hall,  Charanya Ravi, and Marc Levine. We discussed related questions and searched the literature with AI. In this paper, AI was not used to generate ideas for the constructions or prove results. Candace Bethea was supported by National Science Foundation Award DMS-2402099. Kirsten Wickelgren was supported by National Science Foundation Award DMS-2405191. We warmly thank the Simons Laufer Mathematical Sciences Institute (SLMath) for providing a welcoming and productive environment; this material is based upon work supported by the National Science Foundation Grant no. DMS-2424139 while the authors were in residence at SLMath in Berkeley, California during the Fall 2026 semester. 

\section{Construction of the Pushforward}\label{section:pshforward}

Let $k$ be a field. 
 $X$ and $Y$ will always be Noetherian, separated Deligne--Mumford stacks that are tame and finite type over $k$ with projective coarse moduli spaces. 

Let $D_{qc}(X)$ denote the derived category of complexes of quasi-coherent sheaves on $X$, for a construction and general properties see \cite[Section 1]{hall-rydh}. Let $D^+(X)$ denote the derived category of complexes which are bounded below, and $D^\flat(X)$ denote the derived category of bounded complexes. Let $D_c^+(X)$ denote the derived category of bounded below complexes of quasi-coherent sheaves with coherent cohomology, and let $D^{\flat}_c(X)$ denote the derived category of bounded complexes of quasi-coherent sheaves with coherent cohomology. 
Let $D_{\perf}(X)$ denote the derived category of perfect complexes on $X$. 

We recall the results on Serre duality for Deligne--Mumford stacks from \cite{nironi} that will be used in this work. See also work of \cite{lev, hall-rydh} for duality in other contexts. Let $f: X \to Y$ be a proper morphism. Since $f$ is proper, $f$ is separated and quasi-compact. 
Thus by \cite[Theorem 1.16]{nironi} and \cite[Proposition 1]{lev}
\[
Rf_*: D^+(X) \to D^+(Y)
\] admits a right adjoint $f^!$. In particular, there is a trace map 
\[
\Tr_f: Rf_* f^! \calG \to \calG
\] for any $\calG$ in $D^+(Y)$. We obtain an induced map
\begin{equation}\label{eq:SD}
Rf_* R\intHom (\calF, f^! \calG) \to R\intHom(Rf_* \calF, \calG)
\end{equation} by the composition of the canonical map $Rf_* R\intHom (\calF, f^! \calG) \to R\intHom (Rf_* \calF, Rf_* f^! \calG)$ with the trace. Serre duality for stacks is the statement that \eqref{eq:SD} is an equivalence \cite[Corollary 2.10]{nironi} for $\calF$ in $D_c^+(X)$. While we only need an adjoint $f^!$ for the bounded $Rf_*\colon D^+(X)\to D^+(Y)$, there is an adjunction for unbounded $R(_{qc})_*\colon D_{qc}(X)\to D_{qc}(Y)$, see \cite[Theorem 4.14]{hall-rydh} for concentrated morphisms.

Suppose $\pi_Y: Y \to \Spec k$ is tame, proper, and smooth of dimension $d_Y$ with a projective coarse moduli space. Then $\pi_Y^! \calO_k \simeq \omega_Y [d_Y]$ by \cite[Theorem 2.22]{nironi}, where $\omega_Y$ is the canonical sheaf of $Y$, which is invertible. When $Y$ has a projective coarse moduli space, this is also \cite[Theorem 1]{lev}.
\begin{lemma}\label{lem:triangle}
    Let $f: X \to Y$ be a smooth morphism between tame Deligne--Mumford stacks that are proper and smooth over $\Spec k$ of dimension $d_X$ and $d_Y$ respectively. Assume $X$ and $Y$ have projective coarse moduli spaces. Then $f^! \calO_Y \simeq \omega_f [d_f]$ for an invertible sheaf $\omega_f$ on $X$ and $d_f=d_X -d_Y$. 
\end{lemma}

\begin{proof}
   
        Let $\pi_X\colon X\to \Spec k$ and $\pi_Y\colon Y\to \Spec k$ be the structure maps for $X$ and $Y$ respectively so that $\pi_X = \pi_Yf$. 
        In particular, 
        \[
        \omega_X[d_X]\simeq \pi^!_X\calO_k \simeq f^!\pi_Y^! \calO_k \simeq f^!(\omega_Y[d_Y]) \simeq f^!\omega_Y[d_Y]. 
        \]
        Thus $f^!\omega_Y\simeq\omega_X[d_f]$. Since $f^!\omega_Y\simeq f^*\omega_Y\otimes f^!\calO_Y$ \cite[Ideal Theorem a.2]{hartshorne_residues_1966}, $f^!\calO_Y\simeq f^!\omega_Y\otimes f^*\omega_Y^{-1}$. Taking the determinant of the short exact sequence 
        \[
        0\to f^*\Omega_Y \to \Omega_X \to \Omega_f\to 0
        \]
yields the isomorphism $\omega_f \cong \omega_X\otimes f^*\omega_Y^{-1}$ as line bundles. Thus 
        \[
        f^!\calO_Y\simeq f^!\omega_Y\otimes f^*\omega_Y^{-1} \simeq \omega_X[d_f]\otimes f^*\omega_Y^{-1} \simeq \omega_X\otimes f^*\omega_Y^{-1}[d_f] \simeq \omega_f[d_f], 
        \]
        finishing the proof. 
\end{proof}

\begin{definition}\label{bl_naive} For an invertible sheaf $\calL$ on $X$ and integer $n$, let $\Bl^{\naive}(X,\calL[n])$ denote the monoid of isomorphism classes of symmetric bilinear forms
\[
\beta: C_* \otimes C_* \to \calL[n] 
\] under $\oplus$, with $C_*$ in $D_c^\flat(X)$, which are unimodular in the sense that the adjoint map
\[
C_* \to R\intHom(C_*,\calL[n] )
\] is an isomorphism. Define $\Bl^{\naive}_{\mathrm{perf}}(X,\calL[n])$ analogously, with the requirement that $C_*$ is in $D_{\mathrm{perf}}(X)$. See \cite[Section 2.2]{BW-Euler}. 
\end{definition}

Given a morphism $f\colon X\to Y$ as in Lemma \ref{lem:triangle}, our pushforward (Theorem \ref{thm:pushforward}) will be valued in naive bilinear form groups. We ultimately seek to obtain quadratically enriched enumerative results valued in the Grothendieck--Witt ring of $k$, $ \GW(k)$, and we do so by composing the pushforward with the map $\theta$ defined below. 

\begin{definition}\label{df:BlnaivektoGW}
There is a canonical map $\theta\colon \Bl^{\naive}(k, \calO_k) \to \GW(k)$ defined by taking 
\[
\beta: C_* \otimes C_* \to k
\] to $\sum_i (-1)^i \beta_i$, where  each 
\[
\beta_i: V_i \otimes V_i \to k 
\] is defined to be the pairing induced by $\beta$ on 
\[
V_i =
\begin{cases}
(H_i(C_*) \oplus H_{-i}(C_*)), \text{if } i \neq 0 \\
H_0(C_*), \text{if } i = 0. 
\end{cases}
\]

\end{definition}

Serre duality for Deligne--Mumford stacks constructs the following pushforward. See also \cite{CH09}, \cite[Section 2]{KLSW-rational-curves} for analogous constructions in the context of schemes or scheme theoretic loci in stacks.

\begin{theorem}\label{thm:pushforward}
    Let $f: X \to Y$ be a morphism between tame Deligne--Mumford stacks that are proper and smooth over $\Spec k$ of dimension $d_X$ and $d_Y$ respectively, assume $X$ and $Y$ have projective coarse moduli spaces. Then for any invertible sheaf $\calL$ on $Y$ there is a pushforward 
\[
f_*: \Bl^{\naive}(X,f^*\calL \otimes \omega_f [d_f+n]) \to \Bl^{\naive}(Y,\calL[n])
\] defined by sending \[
\beta: C_* \otimes C_* \to f^*\calL \otimes \omega_f [d_f+n]
\] in $\Bl^{\naive}(X,\calL \otimes \omega_f [d_f+n])$ to 
\[
 Rf_* C_* \otimes  Rf_* C_* \to  Rf_*(C_* \otimes C_*) \xrightarrow{Rf_*\beta} Rf_* (f^*\calL \otimes \omega_f [n+d_f]) \simeq \calL \otimes Rf_*(\omega_f [n+d_f]) \xrightarrow{\Tr_f} \calL[n].
\] 
\end{theorem}

\begin{proof}
The morphism $Rf_* C_* \otimes  Rf_* C_* \to  Rf_*(C_* \otimes C_*)$ comes from the lax monoidal structure of $Rf_*$. 

Since $X$, and $Y$ are proper, so is $f$. Thus $ Rf_* C_*$ is in $D^{\flat}_c(Y)$ by \cite[Theorem 15.6.iv]{laumon-moret-bailly}. 
    Let $\beta^{\dagger}: C_* \to \intHom(C_*, f^*\calL \otimes \omega_f [d_f+n])$ be the isomorphism adjoint to $\beta$. It follows that $Rf_* \beta^{\dagger}$ is also an isomorphism. The map $ Rf_* C_*  \to \intHom(Rf_* C_*, \calL[n]) $ adjoint to $f_* \beta$ is the composition of $Rf_* \beta^{\dagger}$ with
    \begin{align*}
    Rf_*\intHom(C_*, f^*\calL \otimes \omega_f [d_f+n]) & \simeq \calL[n] \otimes Rf_*\intHom(C_*, \omega_f [d_f]) \\ &\xrightarrow{\simeq} \calL [n] \otimes \intHom(C_*, \calO_Y) \simeq \intHom(C_*,\calL [n])
    \end{align*}
    where the second equivalence is Serre duality \eqref{eq:SD}, $d_f = d_X-d_Y$, and the first and third are projection formulas. It follows that $f_* \beta$ is unimodular. It is clear that $f_* \beta$ is symmetric when $\beta$ is. 
\end{proof}

We also define the pushforward along a morphism that is relatively oriented.

\begin{definition}\label{defn:relatively_oriented}
   Suppose $f\colon X\to Y$ is a quasi-smooth morphism of Deligne--Mumford stacks over $\Spec(k)$, and let $\omega_f \to X$ denote the relative canonical sheaf. Let $\calL'\to Y$ be a line bundle on $Y$. We say $f$ is \emph{relatively oriented with respect to $\calL'$} if there exists an isomorphism 
   \begin{equation}\label{eq:rel_orient_eq}
   \omega_f\cong \calL^{\otimes 2} \otimes f^*\calL'
   \end{equation}
   for some line bundle $\calL\to X$. We call $\calL$ the \emph{orienting line bundle} and the isomorphism \eqref{eq:rel_orient_eq} the \emph{orientation} of $f$. We say that $X$ is \emph{oriented} or \emph{relatively oriented} if the structure map $\pi_X\colon X\to \Spec k$ is relatively oriented, necessarily with respect to the trivial bundle on $\Spec k$. When $i\colon Z\hookrightarrow X$ is the inclusion of a closed subscheme, we say $Z$ is relatively oriented with respect to $\calL'\to X$ when $i$ is. 
\end{definition}

Suppose $f$ is relatively oriented, quasi-smooth, and additionally satisfies the hypotheses of Theorem \ref{thm:pushforward}. There is an oriented pushforward obtained by composing the pushforward of Theorem \ref{thm:pushforward} with the orientation $\rho$. 

\begin{definition}\label{def:oriented_pushforward} 
Let $f\colon X\to Y$ be morphism of tame Deligne--Mumford stacks that are proper and smooth over $\Spec k$ of dimensions $d_X$ and $d_Y$ respectively, with projective coarse moduli spaces. Let $d_f = d_X-d_Y$ and suppose $f$ is relatively oriented with respect to $\calL'\to Y$ by 
\[
\omega_f\stackrel{\rho}{\cong} \calL^{\otimes 2} \otimes f^*\calL'.
\]
The \emph{oriented pushforward of $f$}, denoted by $f^{\orient}_*$ is given by the composition 
\begin{align}\label{eq:oriented_pshfwd}
\Bl^{\naive}\left(X,f^*\calL'[d_f+n]\right) &\cong \Bl^{\naive}\left(X, (f^*\calL'\otimes \calL^{\otimes 2})[d_f+n]\right) \\
&\stackrel{\rho}{\cong} \Bl^{\naive}\left(X, \omega_f[d_f+n]\right) \stackrel{f_*}{\longrightarrow} \Bl^{\naive}\left(Y, \calO[n]\right), \notag
\end{align} where the first isomorphism is  $\beta\mapsto \beta\otimes \calL$ 
and $f_*$ is the pushforward from Theorem \ref{thm:pushforward}. 
\end{definition}

As is the case for the unoriented pushforward, the oriented pushforward can also be twisted by the pullback of another line bundle $\calM\to Y$, 
\begin{align}\label{eq:twisted_oriented_pushforward}
\Bl^{\naive}\left(X,f^*\calM\otimes f^*\calL'[d_f+n]\right) &\cong \Bl^{\naive}\left(X, (f^*\calM\otimes f^*\calL'\otimes \calL^{\otimes 2})[d_f+n]\right) \\
&\stackrel{\rho}{\cong} \Bl^{\naive}\left(X, f^*\calM\otimes \omega_f\otimes [d_f+n]\right) \stackrel{f_*}{\longrightarrow} \Bl^{\naive}\left(Y,  \calM[n]\right), \notag
\end{align}
Note that $f^{\orient}_*$ depends on a choice of orientation, and also on another choice of line bundle $\calM\to Y$ in the twisted case.

    Note finally that for an invertible sheaf $\calL$ on $Y$ and $f: X \to Y$ a morphism, there is a pullback
\begin{equation}\label{eq:pullbackBlnaive}
f^*: \Bl^{\naive}_{\mathrm{perf}}(Y, \calL[n]) \to \Bl^{\naive}_{\mathrm{perf}}(X, f^*\calL[n])
\end{equation} defined by 
\[
\left[C_*\otimes C_* \stackrel{\beta}{\longrightarrow} \calL[n]\right] 
\mapsto 
\left[ f^*C_*\otimes f^*C_* \stackrel{f^*\beta}{\longrightarrow} f^*\calL[n]\right].
\]

\section{Quadratically Enriched Euler Numbers for Deligne--Mumford stacks} \label{section:GW_invaraints}

\subsection{Quadratically Enriched Euler numbers}
Let $\pi_X: X \to \Spec k$ be a smooth, proper, tame Deligne--Mumford stacks of dimension $d_X$. 

Let $V$ be a locally free sheaf on $X$ of rank $r$. Let $\sigma\colon X\to V$ be any section. 

\begin{definition}\label{def:naive_euler}
    Define the {\em Euler class} of $V$, denoted $e(V, \sigma) \in \Bl^{\naive}(X, \det V [r] )$, to be the bilinear form represented by 
    \[
    \Kos (V) \otimes  \Kos (V)  \to  \det V [r]
    \] on the Koszul complex of $V $,
    \begin{align*}
    & \Kos (V) =\wedge^{r}V^* \to \wedge^{r-1}V^* \to \ldots \wedge^2 V^* \to V^* \to \mathcal{O}_{X}
    \end{align*} in degrees $[0,r]$, defined by the canonical maps
    \[
    \wedge^{i}V^* \otimes \wedge^{r-i}V^* \to \wedge^{r}V^*.
    \]
\end{definition}

We say that $V$ is {\em relatively oriented} by the data of a line bundle $\calL$ on $X$ and isomorphism $\rho: \calL^{\otimes 2} \otimes \det V \xrightarrow{\cong} \omega_X$. When $V\to X$ is relatively oriented by $\calL$ and $\rho$, we can define the quadratically enriched Euler number of $V$. 

\begin{definition}\label{def:euler_number}
Assume $V\to X$ is a relatively oriented vector bundle and $r = d_X$. Let $\pi_X\colon X\to \Spec k$ be the structure map for $X$. Assume $X$ has a projective coarse moduli space. Define the \emph{quadratically enriched Euler number} of $V$, $n(V,\sigma)$, to be the image of $e(V,\sigma)$ under the composition 
\[
 \Bl^{\naive}(X, \det V[r]) \cong \Bl^{\naive}(X, \omega_X [d_X]) \stackrel{(\pi_X)_*}{\to} \Bl^{\naive}(k)\stackrel{\Theta}{\to} \GW(k)
\]
where where the first isomorphism is given by the orientation data $\rho$, $\theta$ is defined in Definition \ref{df:BlnaivektoGW}, and $(\pi_X)_*$ is defined in Defition~\ref{thm:pushforward}. 
\end{definition}

\begin{remark}
    Note that $n(V,\sigma)$ is independent of choice of section by the proof of \cite[Proposition 2.3]{BW-Euler}. 
\end{remark}
In light of this, we henceforth denote the quadratically enriched Euler number of $V$ by $n(V)$ with no ambiguity. 

We now use the quadratically enriched Euler number to define a (Grothendieck--)Witt invariant, in the sense of Serre \cite{GMS}, in a case of particular interest. Given a projective variety $X$ over a field $k$ and $\beta\in H_2(X,\Z)$, let $\Mgnbar(X,\beta)$ denote the moduli space of genus $g$, $n$-marked stable maps to $X$ in class $\beta$.

\begin{exa}
    $\overline{M}_{0,n}(X,\beta)$ is smooth by \cite[1.3.2]{kontsevich-enumeration} when $X$ is convex in the sense of \cite[2.4.2]{KM-recursive}. For example, $\mathbb{P}^r_k$ is convex for any $r$ and any field $k$, see \cite[Theorem 7.1.4]{cox-katz-GW-book} and \cite{behrend-manin}. 
\end{exa}

\begin{exa}As is standard, we will write $\Mzeronbar(\bbP^r_k,d)$ to mean the moduli stack parameterizing stable maps from nodal, rational curves whose degree is $d\cdot \ell$ for $\ell$ the class of a line in $\bbP^r_k$, passing through $n$ distinguished points. 
\end{exa}

To define a (Grothendieck--)Witt invariant and Euler number for $\Mgnbar(X,\beta)$ using the pushforward of Theorem \ref{thm:pushforward}, we impose the following assumption henceforth: 

\begin{assumption}\label{assump:conditions_Mgn}
    Let $k$ be a field and let $X$ be a smooth, proper variety over $k$ such that $\overline{M}_{g,n}(X,\beta)$ is smooth, tame, and has a projective coarse moduli space given a fixed $\beta\in H_2(X, \mathbb{Z})$. 
\end{assumption}

Assumption \ref{assump:conditions_Mgn} holds, for example, over a field $k$ of characteristic 0 for $\overline{M}_{g,n}(\mathbb{P}^r_k,d)$ and for $\overline{M}_{g,n}(X,d)$ when $X\subseteq \mathbb{P}^r_k$ is a projective scheme of finite presentation, see \cite[Theorem 2.8, Section 2.5]{abramovich-oort}.  It is worth noting that $\overline{M}_{g,n}(\mathbb{P}^r_k,d)$ and $\overline{M}_{g,n}(X,d)$ have projective coarse moduli spaces over general fields by \cite{abramovich-oort}, but they may not be tame in all characteristics, see for example \cite[Proposition 5.1]{alper-slicing}. 

Let $L_1, \ldots, L_m$ be finite, separable extensions of $k$ such that $\sum_{i=1}^m [L_i:k] = n$. Let $L=\{L_1, \ldots, L_m\}$. We define the smooth, proper, tame DM stack
\begin{equation}\label{eq:twistedpi}
\twistpi\colon \twistMgnbarXb \to \Spec k
\end{equation}
as the twist of $\MgnbarXb$ given by the action of $\op{Gal}(k^{\op{sep}}/k)$ on $\MgnbarXb \times_k \Spec \overline{k}$ not only by the canonical action on $\overline{k}$ but also by simultaneously permuting the $n$ marked points in the same way as the $n$ homomorphisms $L =\prod_{i=1}^m L_i \to \overline{k}$. See \cite[4.1.4]{KLSW-rational-curves} for more information. 

The Euler numbers $n(\twistMgnbarXb)$ in $\GW(k)$ form a (Grothendieck--)Witt invariant when Assumtion \ref{assump:conditions_Mgn} holds in the following sense. Let $\Et_n$ be the functor from fields over $k$ to sets which takes a field $K$ to the set $\Et_n(K)$ of finite \'etale algebras $K \to L$ with $[L:K]=n$. Consider $\GW$ also to be a functor between these two categories taking $K$ to the set underlying the Grothendieck--Witt ring $\GW(K)$. As in Serre's work in \cite{GMS}, define a Grothendieck--Witt invariant to be a natural transformation 
\[
\Et_n \to \GW.
\]

\begin{definition}
Let $k$, $X$, and $\Mgnbar(X,\beta)$ be as in Assumption \ref{assump:conditions_Mgn}. Let $$\chi(\MgnbarXb): \Et_n \to \GW$$ be the \textit{Grothendieck--Witt invariant} taking $K \to L$ with $[L:K]=n$ to 
    \[
    n(T\twistMgnbarXb) \in \GW(K).
    \]
\end{definition}

This defines a quadratically enriched Euler number of $\twistMgnbarXb$. 

By \cite[Theorems 6.4.16, 6.5.23]{Alper-book}, $\Mgnbar$ is a smooth, proper, Deligne--Mumford stack for $2g-2+n>0$. Note then that it is also tame in characteristic $0$. We my thus similarly define the Grothendieck--Witt invariant
\[
\chi(\Mgnbar): \Et_n \to \GW
\] by taking $K \to L$ with $[L:K]=n$ to 
    \[
    n(T\twistMgnbar) \in \GW(K).
    \]

This gives a quadratic enrichment of the Euler characteristics of $\MgnbarXb$ and $\Mgnbar$. See also \cite{BRW} for a Witt-invariant description of quadratically enriched Gromov--Witten or Welschinger invariants. Euler characteristics of the moduli space of curves and stable maps are of great interest in enumerative geometry. See for example work of Chan--Faber--Galatius--Payne, Chan--Galatius--Payne, Kannan--Song, Getzler--Kapranov, and Harer--Zagier \cite{CFGP, chan-galatius-payne, kannan-song, getzler-kapranov, harer-zagier}.

\section{Quadratically Enriched Quantum K-Invariants}

\subsection{Fundamental class of relatively oriented closed subschemes}\label{Subsection:FundClassSubLNew}

Suppose that $X$ is a smooth, projective scheme over $k$ of dimension $d_X$. Let $k \subseteq L$ be a finite separable extension and let $X_L$ denote the basechange $X_L = X \times_k \Spec L$. Let $\Res^L_{k}$ denote the Weil restriction. Suppose $Z \hookrightarrow X_L$ is a closed subscheme of codimension $c$ with $Z$ smooth. Then we obtain a canonical map
\[
i:\Res^L_{k} Z \hookrightarrow \Res^L_{k} X.
\] Let $r=c[L:k]$. Suppose that $i$ is relatively oriented with respect to a line bundle $\calL'$ on $\Res^L_{k} X$, i.e., $$\omega_i\stackrel{\rho}{\cong} \calL^{\otimes 2}\otimes i^*\calL'$$ for some $\calL\to Z$. In \eqref{eq:twisted_oriented_pushforward} we can take $\calM = \calL'$ and $n=r$, and note that there is a canonical isomorphism
\[
\Bl^{\naive}(Z, i^* \calM \otimes f^*\calL'[d_i+r]) \cong \Bl^{\naive}(Z, \calO_Z)
\]
giving an oriented pushforward
 \begin{equation}\label{ioriented_pushforward_for_FundClassZ_def}
     i_*^{\orient}: \Bl^{\naive}(Z, \calO_Z) \to \Bl^{\naive}(X, \calL'[r]).
     \end{equation}
     
\begin{definition}\label{df:fundclass[Zrho]}
     Define the {\em fundamental class of ($Z$, $\rho$)} $$[Z,\rho] \in \Bl^{\naive}(\Res^L_{k} X, \calL' [r])$$ to be the oriented pushforward \eqref{ioriented_pushforward_for_FundClassZ_def} of $
    \calO_Z \times \calO_Z \to \calO_Z $ given by
     \[
     [Z,\rho] = i_*^{\orient}(\calO_Z \times \calO_Z \to \calO_Z).
     \]
    \end{definition}

See also \cite{Zhang-res} for another related and prior definition of the fundamental class for a subvariety $Z$ defined over $L$. 

\begin{exa}\label{exa:ratnl_pt_fund_class}
Consider the rational point $p=[0:0:1]$ of $\mathbb{P}^2_k$ corresponding to a closed immersion
    \[
    i: \Spec k \hookrightarrow \mathbb{P}^2.
    \]
    We will compute the fundamental class $[p]$ of $p$, $[p,\rho] = i_*^{\orient}(\calO_p\otimes\calO_p\to\calO_p)$, where $i_*^{\orient}$ is as in Definition~\ref{def:oriented_pushforward}, and $\rho$ is defined as follows. Let $\calL' = \omega_{\bbP^2}^{-1} \cong \calO(3)$ and $\calL = \calO_k$. There is an orientation of $i$ given by 
    \[
    \omega_i\cong \omega_p \otimes i^*\omega_{\bbP^2}^{-1} \cong \calO_p\otimes i^*\omega_{\bbP^2}^{-1} = \calL^{\otimes 2}\otimes i^*\calL'
    \]
    that follows from the short exact sequence 
    \[
    0\to i^*\Omega_{\bbP^2} \to \Omega_p \to \Omega_i\to 0. 
    \]
    Let $\calM = \omega_{\bbP^2} \cong \calO(-3)$. 

    Consider the bilinear form on $\Spec k$ given by multiplication, $ \calO_p\otimes \calO_p \stackrel{\mu}{\to} \calO_p$. The fundamental class of $p$ is the bilinear form on $\mathbb{P}^2$ given by the twisted, oriented pushforward of $\mu$ along \eqref{eq:twisted_oriented_pushforward},
    \begin{align}
i_*^{\orient}\colon \Bl^{\naive}\left(k,i^*\calM\otimes i^*\calL'[-2+2]\right) &\cong \Bl^{\naive}\left(k, (i^*\calM\otimes i^*\calL'\otimes \calL^{\otimes 2})[-2+2]\right) \\
&\stackrel{\rho}{\cong} \Bl^{\naive}\left(k, i^*\calM\otimes \omega_i [-2+2]\right) \stackrel{i_*}{\longrightarrow} \Bl^{\naive}\left(\bbP^2,  \calM[2]\right), \notag
\end{align}
where the first equivalence is $\mu\mapsto \mu\otimes \calL$, the second is given by the orientation $\rho$, and $i_*$ is the pushforward of Theorem \ref{thm:pushforward}. In particular, the fundamental class $[p, \rho]$ is given by the bilinear form 
\begin{equation}\label{eq:fund_class_of_pt}
    i_*\calL\otimes i_*\calL \to i_*\calL \cong i_*(i^*\calM \otimes i^*\calL'\otimes \calL^{\otimes 2}[-2+2]) \stackrel{\rho}{\cong} \calM [2]\otimes i_* \omega_i[-2] \xrightarrow{\calM[2]\otimes \Tr_i}  \calM[2]. 
\end{equation}
Recall $\omega_i\simeq i^! \calO_{\bbP^2}$. 

 The pushforward $i_*\calO_p$ is the Koszul complex $\Kos(p)$, which is resolved by 
        \begin{equation}
        \begin{tikzcd}
            0\arrow{r} & \calO(-2) \arrow{r}
            {(-yx)^T} & \calO(-1)\oplus \calO(-1)\arrow["(xy)"]{r} & \calO_{\bbP^2}\arrow{r} & 0. 
        \end{tikzcd}
    \end{equation}
Thus $i_*\calO_p\otimes i_*\calO_p\xrightarrow{i_*\mu} i_*\calO_p$ is equivalent to the bilinear form on the Koszul complex $$\Kos(p)\otimes \Kos(p) \to \Kos(p),$$ whence $i_*\calL\otimes i_*\calL \xrightarrow{i_*(\calL\otimes \mu)} i_*\calL$ is also equivalent to the bilinear form $\Kos(p)\otimes \Kos(p) \to \Kos(p)$ since $\calL = \calO_k$. Thus to describe the fundamental class $[p, \rho]$ in \eqref{eq:fund_class_of_pt}, we must understand the equivalence $\Kos(p)\cong i_*(i^*\calM\otimes i^*\calL'\otimes \calL^{\otimes 2})\stackrel{\rho}{\cong} \calM [2]\otimes i_* \omega_i[-2]$ and the map $\calM[2]\otimes \Tr_i$. 

We start with $\calM[2]\otimes \Tr_i$. Note that since $i$ is finite, 
    \[
    Ri_*\omega_i[-2] \simeq i_* \omega_i[-2] \simeq Ri_* i^! \calO_{\mathbb{P}^2}, 
    \]
    which is equivalent to 
    \begin{center}
        \begin{tikzcd}
            0\arrow{r} & \calO_{\bbP^2} \arrow{r}            {(xy)^T}& \calO(1)\oplus \calO(1)\arrow["(-yx)"]{r} & \calO(2)\arrow{r} & 0
        \end{tikzcd}
    \end{center}
    with $\calO_{\bbP^2}$ in degree 0. 
    Using this resolution, taking the trace $\Tr_i\colon Ri_* i^! \calO_{\bbP^2}\to \calO_{\bbP^2}$ is equivalent to the map on chain complexes
    \begin{center} 
    \begin{tikzcd} 
    0\arrow{r} \arrow{d} & \calO_{\bbP^2}\arrow{d}{\op{id}} \arrow{r}& \calO(1)\oplus \calO(1)\arrow{r} \arrow{d} & \calO(2) \arrow{r} \arrow{d} & 0 \arrow{d} \\
    0 \arrow{r} & \calO_{\bbP^2} \arrow{r} & 0 \arrow{r} & 0 \arrow{r} & 0 
    \end{tikzcd}
    \end{center}
    Tensoring with $\calM[2]\cong \calO(-3)[2]$, the map $\calM[2]\otimes \Tr_i\colon \calM[2]\otimes Ri_*\omega_i[-2]\to \calM[2]$ is
    \begin{equation}\label{eq:tr_tensor_M}
    \begin{tikzcd}
       0\arrow{r} \arrow{d} & \calO(-3)\arrow{d} \arrow{r}& \calO(1)^{\oplus 2}\otimes \calO(-3) \arrow{r} \arrow{d} & \calO(-1) \arrow{r} \arrow{d} & 0 \arrow{d} \\
    0 \arrow{r} & \calO(-3) \arrow{r} & 0 \arrow{r} & 0 \arrow{r} & 0 
    \end{tikzcd}
    \end{equation}
    where $\calO(-3)$ is in degree $-2$.

    We now describe the equivalence $\Kos(p)\cong i_*(i^*\calM\otimes i^*\calL'\otimes \calL^{\otimes 2})\stackrel{\rho}{\cong} \calM [2]\otimes i_* \omega_i[-2]$. 
Note 
\begin{equation}\label{eq:big_koszul_equivalence}
i_*(i^*\calM\otimes i^*\calL'\otimes \calL^{\otimes 2}) \cong \calM \otimes i_*(i^*\calL'\otimes \calL^{\otimes 2}) \stackrel{\rho}{\cong} \calM\otimes i_*\omega_i\cong \calM\otimes i_*i^!\calO_{\bbP^2},
\end{equation}
so we equivalently describe $\calM[2]\otimes i_*i^!\calO_{\bbP^2}\xrightarrow{\sim} \Kos(p)$. Consider the equivalence $$\calM[2]\otimes i_*i^!\calO_{\bbP^2} \stackrel{\sim}{\to}\Kos(p)$$ given by 
    \begin{equation}\label{eq:kos_equivalence}
        \begin{tikzcd}
        0\arrow{r} \arrow{d} & \calO(-3)\arrow{d}{z\mapsto z} \arrow{r}
        {(xy)^T}& \calO(1)^{\oplus 2} \otimes \calO(-3) \arrow{r}{(-yx)} \arrow{d}{(a,b) \mapsto (-bz, az)} & \calO(-1) \arrow{r} \arrow{d}{z\mapsto z } & 0 \arrow{d} \\
    0 \arrow{r} & \calO(-2) \arrow{r}
    {(-yx)^T}& \calO(-1)\oplus\calO(-1) \arrow{r} & \calO_{\bbP^2} \arrow{r}{(xy)} & 0 
        \end{tikzcd}
    \end{equation}
    
    In total, substituting \eqref{eq:tr_tensor_M} and \eqref{eq:kos_equivalence} into \eqref{eq:fund_class_of_pt} and using $i_*\calL\cong \Kos(p)$ shows that the fundamental class $[p, \rho]$, 
    \[
    i_*\calL\otimes i_*\calL \to i_*\calL \cong i_*(i^*\calM \otimes i^*\calL'\otimes \calL^{\otimes 2}[-2+2]) \stackrel{\rho}{\cong} \calM [2]\otimes i_* \omega_i[-2] \xrightarrow{\calM[2]\otimes \Tr_i}  \calM[2]
    \]
    is represented by the bilinear form 
   \begin{equation}
    \begin{tikzcd}
    \Kos(p)\arrow[r, phantom, "\otimes"] & \Kos(p) \arrow{r}{\mu}  & \Kos(p) \arrow[r, phantom, "\simeq"] & \calM[2]\otimes i_*\omega_i[-2]\arrow{r}{\calM[2]\otimes \Tr_i} & \calM[2] \\ 
    0 \arrow{d} & 0 \arrow{d} & 0 \arrow{d} & 0\arrow{d} & 0\arrow{d}\\ 
    \calO(-2)\arrow{d} & \calO(-2)\arrow{d} & \calO(-2) \arrow{d} &  \calO(-3) \arrow{d} & \calO(-3) \arrow{d}\\
    \calO(-1)^{\oplus 2} \arrow{d} \arrow[r, phantom, "\otimes"] & \calO(-1)^{\oplus 2} \arrow{d}\arrow{r}{\mu} & \calO(-1)^{\oplus 2} \arrow{d} \arrow[r, phantom, "\simeq"]  & \begin{array}{c}
    \calO(1)^{\oplus 2}  \\
    \otimes\,\calO(-3)
\end{array}\arrow{d} \arrow{r}{\calM[2]\otimes \Tr_i} & 0 \arrow{d}\\ 
    \calO_{\bbP^2} \arrow{d} & \calO_{\bbP^2} \arrow{d} & \calO_{\bbP^2} \arrow{d} & \calO(-1)\arrow{d} & 0 \arrow{d}\\
    0 & 0 & 0 & 0 & 0\\ 
    \end{tikzcd}
\end{equation}
    in $\Bl^{\naive}(\bbP^2, \calM[2]) = \Bl^{\naive}(\bbP^2, \omega_{\bbP^2}[2])$, as desired. 
    \end{exa}

\begin{exa}\label{exa:non_ratnl_pt_fund_class}
    Let $L$ be a finite, separable extension of $k$. Consider a point $p: \Spec L \to \mathbb{P}^2_L$. Write $Y=\Res_{L/k}\bbP^2_L$. Let 
    \[
    i: \Spec k \to Y
    \]
    be the Weil restriction of $p$ to $k$. Note that $i$ is relatively oriented with respect to the line bundle $\calL' = \omega_{Y}$ on $Y$ by the isomorphism
    \[
    \rho: \omega_i\xrightarrow{\cong} \omega_{k}\otimes i^*\calL' \xrightarrow{\cong} \calO_k^{\otimes 2}\otimes i^*\calL' 
    \] defining a pushforward 
    \[
    i_*^{\orient}: \Bl^{\naive}(Z, \calO_Z) \to \Bl^{\naive}(Y, \omega_{Y}[2[L:k]])
    \] by \eqref{ioriented_pushforward_for_FundClassZ_def}. We thus have fundamental class of $p$ in $\Bl^{\naive}(Y, \omega_{Y}[2[L:k]])$ defined 
    \[
    [p, \rho]= i_*^{\orient}(\calO_k \otimes \calO_k \to \calO_k).
    \] This fundamental class can be represented by an explicit zig-zag of complexes similarly to Example \ref{exa:ratnl_pt_fund_class}. 
\end{exa}

\subsection{Definition of enriched invariants in the smooth case }
We present a definition of a quadratically enriched version of Y.P. Lee's quantum $K$-invariants, see \cite[Section 4.2]{lee-quantum-invariants}, when $\MgnbarXb$ is oriented and Assumption \ref{assump:conditions_Mgn} holds. We do not prove any claims about existence of orientations here, focusing only on a definition of the enriched invariants under (meaningful) assumptions on the existence of orientations. Note that in \cite{KLSW-rel-orientation}, a relative orientation was constructed on a large locus of $\overline{M}_{0,n}(\mathbb{P}^2,d)$, but not on the entire stack. 

As in the previous section, let $X$ be smooth and projective over a field $k$ such that $\MgnbarXb$ is smooth. let $L_1, \ldots, L_m$ be finite, separable extensions of $k$ such that $\sum_{i=1}^m [L_i:k] = n$. Let $L=\{L_1, \ldots, L_m\}$. Let
\begin{equation}
\twistpi\colon \twistMgnbarXb \to \Spec k.
\end{equation} be the twisted moduli stack of stable maps from \eqref{eq:twistedpi}. The total evaluation map $\ev: \MgnbarXb \to X^n$ taking a stable map $(f:C \to X, p_1,\ldots,p_n)$ to $(f(p_1),\ldots, f(p_n))$ may also be twisted, resulting in maps
\begin{equation}\label{eq:twistedevals}
    \twistev_i\colon \twistMgnbarXb \to \Res_{L_i/k} X_{L_i} \quad \text{and} \quad \twistev\colon \twistMgnbarXb\to \prod_{j=1}^m \Res_{L_j/k}X_{L_j}
\end{equation}
as in \cite[4.1.4]{KLSW-rational-curves}. To see this, note that the restriction of scalars $$\prod_{j=1}^m \Res_{L_j/k}X_{L_j} \cong \Res_{L/k} X$$ for $L = \prod_{j=1}^m L_j$ is isomorphic to the twist of $X^m$, where the Galois action on $X^n$ permutes the $m$ factors of $X$ with the same action as on the $m $ ring homomorphisms $L \to \overline{k}$ over $k$. It follows that the map $\ev \otimes_k \overline{k}$ is equivariant for the twisted action.

Let $Z_j \hookrightarrow X_{L_j}$ be smooth, closed subschemes of codimension $c_j$ for $j=1,\ldots,m$. As in Section \ref{Subsection:FundClassSubLNew}, define $r_j = c_j[L_j:k]$ and suppose that each 
\[
i_j: \Res^L_k Z \to \Res^L_k X
\] is relatively oriented by $\rho_j: \omega_{i_j} \to \calL_j^{\otimes 2} \otimes i_j^* \calL'_j$ with respect to a line bundle $\calL'_j$ on $\Res^L_k X$. By Definition~\ref{df:fundclass[Zrho]}, we then have $[Z_j,\rho_j] \in \Bl^{\naive}(X, \calL'_j[r_j])$. See Examples \ref{exa:ratnl_pt_fund_class} and \ref{exa:non_ratnl_pt_fund_class} when $Z_j$ is a point.

Now suppose $\sum_{j=1}^m r_j = \dim \twistMgnbarXb=d$. Assume there is an orientation of $\twistpi$ relative to $(\twistev)^* \boxtimes_{j=1}^m \calL'_j $ given by 
\begin{equation}\label{eq:orient_twist_omega}
\omega_{\twistpi} \cong \left((\twistev)^* \boxtimes_{j=1}^m \calL'_j  \right)\otimes \calL^2 
\end{equation}
for some $\calL\to \twistMgnbarXb$. Then we have an oriented pushforward 
\begin{equation}\label{eq:oriented_twistpi}
(\twistpi)^{\orient}_*\colon \Bl^{\naive}\left(\twistMgnbarXb, \left((\twistev)^* \boxtimes_{j=1}^m \calL'_j\right) [d]\right) \to \Bl^{\naive}(k, \calO_k). 
\end{equation}

\begin{definition}\label{defn:smooth_quantum_K}
Assume $X$, $k$, $\beta$, and $\Mgnbar(X,\beta)$ are as in Assumption \ref{assump:conditions_Mgn} and assume $\MgnbarXb$ is oriented. 
Define the $n$-marked, genus $g$, \emph{quadratically enriched quantum $K$-invariant} of $X$ to be 
    \begin{equation}
        \langle Z_1, \ldots, Z_m\rangle^{X_{/k},L}_{g,\beta} :=  \theta\circ(\twistpi)^\orient_*\left(\cup_{i=1}^m (\twistev_i)^*[Z_i]  \right) 
    \end{equation}
    in $\GW(k)$, where $\theta$ is as in Definition \ref{df:BlnaivektoGW}, $(\twistpi)^\orient_*$ is the oriented pushforward of $\twistpi$ of \eqref{eq:oriented_twistpi}, and $(\twistev_i)^*$ is the pullback (see \eqref{eq:pullbackBlnaive}). 
\end{definition}

This defines the first quadratically enriched quantum $K$-invariant of a smooth, projective scheme over a general field $k$ when $\MgnbarXb$ is oriented.

\newpage 

\bibliographystyle{amsalpha}
\bibliography{bib_equiv_GW}

@article {hall-rydh,
    AUTHOR = {Hall, Jack and Rydh, David},
     TITLE = {Perfect complexes on algebraic stacks},
   JOURNAL = {Compos. Math.},
  FJOURNAL = {Compositio Mathematica},
    VOLUME = {153},
      YEAR = {2017},
    NUMBER = {11},
     PAGES = {2318--2367},
      ISSN = {0010-437X,1570-5846},
   MRCLASS = {14F05 (14A20 18E30)},
  MRNUMBER = {3705292},
MRREVIEWER = {Stefan\ Schr\"oer},
       DOI = {10.1112/S0010437X17007394},
       URL = {https://doi.org/10.1112/S0010437X17007394},
}

@article{27lines,
   title={An arithmetic count of the lines on a smooth cubic surface},
   volume={157},
   ISSN={1570-5846},
   url={http://dx.doi.org/10.1112/S0010437X20007691},
   DOI={10.1112/s0010437x20007691},
   number={4},
   journal={Compositio Mathematica},
   publisher={Wiley},
   author={Kass, Jesse and Wickelgren, Kirsten},
   year={2021},
   month=apr, pages={677–709} }

@article{levineaspects,
   title={Aspects of Enumerative Geometry with Quadratic Forms},
   volume={25},
   ISSN={1431-0643},
   url={http://dx.doi.org/10.4171/dm/797},
   DOI={10.4171/dm/797},
   journal={Documenta Mathematica},
   publisher={European Mathematical Society - EMS - Publishing House GmbH},
   author={Levine, Marc},
   year={2020},
   pages={2179–2239} }

@article{KLSW-rational-curves, 
   title={A quadratically enriched count of rational curves},
   number={},
   journal={arXiv: 2307.01936},
   author={J. Kass and M. Levine and J. Solomon and K. Wickelgren},
   year={2023},
   pages={0-47} }

@article {BW-Euler,
    AUTHOR = {Bachmann, Tom and Wickelgren, Kirsten},
     TITLE = {Euler classes: six-functors formalism, dualities, integrality
              and linear subspaces of complete intersections},
   JOURNAL = {J. Inst. Math. Jussieu},
  FJOURNAL = {Journal of the Institute of Mathematics of Jussieu. JIMJ.
              Journal de l'Institut de Math\'ematiques de Jussieu},
    VOLUME = {22},
      YEAR = {2023},
    NUMBER = {2},
     PAGES = {681--746},
      ISSN = {1474-7480,1475-3030},
   MRCLASS = {14F42 (19E15 55R40)},
  MRNUMBER = {4557905},
MRREVIEWER = {Satoshi\ Mochizuki},
       DOI = {10.1017/S147474802100027X},
       URL = {https://doi-org.proxy.lib.duke.edu/10.1017/S147474802100027X},
}

@article {CH09,
    AUTHOR = {Calm\`es, Baptiste and Hornbostel, Jens},
     TITLE = {Tensor-triangulated categories and dualities},
   JOURNAL = {Theory Appl. Categ.},
  FJOURNAL = {Theory and Applications of Categories},
    VOLUME = {22},
      YEAR = {2009},
     PAGES = {No. 6, 136--200},
      ISSN = {1201-561X},
   MRCLASS = {18D10},
  MRNUMBER = {2520968},
}

@article{KM-recursive, 
   title={Gromov-Witten classes, quantum cohomology, and enumerative geometry},
   number={3},
    volume={164},
   journal={Comm. Math. Phys.},
   author={M. Kontsevich and Yu. Manin},
   year={1994},
   pages={525-562} }

@article{KW-local-degree, 
   title={The class of Eisenbud-Khimshiashvili-Levine is the local $\mathbb{A}^1$-Brouwer degree},
    volume={168},
    number={1},
   journal={Duke Math Journal},
   author={J. Kass and K. Wickelgren},
   year={2019},
   pages={429–469} }

@book{solomon-thesis,
	address = {Ann Arbor, MI},
	author = {J.P. Solomon},
	publisher = {ProQuest LLC},
    series = {Ph.D. Thesis, Massachusetts Institute of Technology}, 
	title = {Intersection theory on the moduli space of holomorphic curves with {L}agrangian boundary conditions},
	year = {2006}}

@article{Hoyois-lefschetz, 
   title={A quadratic refinement of the Grothendieck–Lefschetz–Verdier trace formula},
    volume={14},
    number = {6},
   journal={Algebr. Geom. Topol.},
   author={M. Hoyois},
   year={2014}, 
    pages = {3603-3658} 
}

@article{SW-4-lines, 
   title={An Arithmetic Count of the Lines Meeting Four Lines in $\mathbb{P}^3$},
    volume={374},
    number = {5},
   journal={Trans. Amer. Math. Soc.},
   author={P. Srinivasan and K. Wickelgren},
   year={2021}, 
    pages = {3427--3451} 
}

@article{quadratic-tropical, 
   title={Arithmetic counts of tropical plane curves and their properties},
   journal={arXiv:2309.12586, Accepted in Advances in Geometry},
   author={A. Jaramillo Puentes and H. Markwig and F. Rohrle and S. Pauli},
   year={2023}, 
    pages = {1-30} 
}

@article{quadratic-twisted-cubics, 
   title={Quadratic Counts of Twisted Cubics},
   journal={arXiv: 2206.05729, Accepted in Trends in Mathematics},
   author={M. Levine and S. Pauli},
   year={2022},
    pages = {1-60}
}

@article {Hoyois6ff,
    AUTHOR = {M. Hoyois},
     TITLE = {The six operations in equviariant motivic homotopy theory},
   JOURNAL = {Advances in Mathematics},
    VOLUME = {305},
      YEAR = {2017},
     PAGES = {197-279},
      ISSN = {0001-8708},
       URL = {https://doi.org/10.1016/j.aim.2016.09.031},
}

@misc{Zhang-res, 
   title={Some Enumerative {W}itt Invariants And Twisted {P}ontryagin Classes},
   journal={In preparation},
   author={Ruijie (Jared) Zhang},
   pages = {},
   year={2026},
}

@article {khan-ravi,
    AUTHOR = {Khan, Adeel A. and Ravi, Charanya},
     TITLE = {Generalized cohomology theories for algebraic stacks},
   JOURNAL = {Adv. Math.},
  FJOURNAL = {Advances in Mathematics},
    VOLUME = {458},
      YEAR = {2024},
     PAGES = {Paper No. 109975, 104},
      ISSN = {0001-8708,1090-2082},
   MRCLASS = {14A20 (14C15 14F42 19E08 55N20)},
  MRNUMBER = {4811546},
       DOI = {10.1016/j.aim.2024.109975},
       URL = {https://doi.org/10.1016/j.aim.2024.109975},
}

@incollection {kontsevich-enumeration,
    AUTHOR = {Kontsevich, Maxim},
     TITLE = {Enumeration of rational curves via torus actions},
 BOOKTITLE = {The moduli space of curves ({T}exel {I}sland, 1994)},
    SERIES = {Progr. Math.},
    VOLUME = {129},
     PAGES = {335--368},
 PUBLISHER = {Birkh\"auser Boston, Boston, MA},
      YEAR = {1995},
      ISBN = {0-8176-3784-2},
   MRCLASS = {14N10 (14D22 14L30)},
  MRNUMBER = {1363062},
MRREVIEWER = {Anatoly\ Libgober},
       DOI = {10.1007/978-1-4612-4264-2\_12},
       URL = {https://doi.org/10.1007/978-1-4612-4264-2_12},
}

@incollection {abramovich-oort,
    AUTHOR = {Abramovich, Dan and Oort, Frans},
     TITLE = {Stable maps and {H}urwitz schemes in mixed characteristics},
 BOOKTITLE = {Advances in algebraic geometry motivated by physics ({L}owell,
              {MA}, 2000)},
    SERIES = {Contemp. Math.},
    VOLUME = {276},
     PAGES = {89--100},
 PUBLISHER = {Amer. Math. Soc., Providence, RI},
      YEAR = {2001},
      ISBN = {0-8218-2810-X},
   MRCLASS = {14H10 (14A20)},
  MRNUMBER = {1837111},
MRREVIEWER = {Marian\ Aprodu},
       DOI = {10.1090/conm/276/04513},
       URL = {https://doi.org/10.1090/conm/276/04513},
}

@article{lev,
    AUTHOR = {Levchenko, Denis},
     TITLE = {Serre duality for tame {D}eligne-{M}umford stacks},
   JOURNAL = {Res. Math. Sci.},
  FJOURNAL = {Research in the Mathematical Sciences},
    VOLUME = {9},
      YEAR = {2022},
    NUMBER = {4},
     PAGES = {Paper No. 67, 5},
      ISSN = {2522-0144,2197-9847},
   MRCLASS = {14A20 (14F06)},
  MRNUMBER = {4514209},
MRREVIEWER = {Giulio\ Bresciani},
       DOI = {10.1007/s40687-022-00367-7},
       URL = {https://doi.org/10.1007/s40687-022-00367-7},
}

@article{nironi, 
   title={Grothendieck Duality for Deligne-Mumford Stacks},
   journal={arXiv:0811.1955},
   author={F. Nironi},
   year={2008}, 
    pages={1-35}
}

@book {cox-katz-GW-book,
    AUTHOR = {Cox, David A. and Katz, Sheldon},
     TITLE = {Mirror symmetry and algebraic geometry},
    SERIES = {Mathematical Surveys and Monographs},
    VOLUME = {68},
 PUBLISHER = {American Mathematical Society, Providence, RI},
      YEAR = {1999},
     PAGES = {xxii+469},
      ISBN = {0-8218-1059-6},
   MRCLASS = {14J32 (14-02 14M25 14N10 14N35 32G81 32J81 32Q25)},
  MRNUMBER = {1677117},
MRREVIEWER = {Andreas\ Gathmann},
       DOI = {10.1090/surv/068},
       URL = {https://doi.org/10.1090/surv/068},
}

@article {behrend-manin,
    AUTHOR = {Behrend, K. and Manin, Yu.},
     TITLE = {Stacks of stable maps and {G}romov-{W}itten invariants},
   JOURNAL = {Duke Math. J.},
  FJOURNAL = {Duke Mathematical Journal},
    VOLUME = {85},
      YEAR = {1996},
    NUMBER = {1},
     PAGES = {1--60},
      ISSN = {0012-7094,1547-7398},
   MRCLASS = {14D20 (14C25 14D22)},
  MRNUMBER = {1412436},
MRREVIEWER = {Barbara\ Fantechi},
       DOI = {10.1215/S0012-7094-96-08501-4},
       URL = {https://doi.org/10.1215/S0012-7094-96-08501-4},
}

@book {laumon-moret-bailly,
    AUTHOR = {Laumon, G\'erard and Moret-Bailly, Laurent},
     TITLE = {Champs alg\'ebriques},
    SERIES = {Ergebnisse der Mathematik und ihrer Grenzgebiete. 3. Folge. A
              Series of Modern Surveys in Mathematics [Results in
              Mathematics and Related Areas. 3rd Series. A Series of Modern
              Surveys in Mathematics]},
    VOLUME = {39},
 PUBLISHER = {Springer-Verlag, Berlin},
      YEAR = {2000},
     PAGES = {xii+208},
      ISBN = {3-540-65761-4},
   MRCLASS = {14A20 (14D20)},
  MRNUMBER = {1771927},
MRREVIEWER = {Dan\ Edidin},
}

@article{KLSW-rel-orientation, 
   title={A relative orientation for the moduli space of stable maps to a del Pezzo surface},
   journal={arXiv:2307.01941, to appear in Algebraic Geometry},
   author={J. Kass and M. Levine and J. Solomon and K. Wickelgren},
   year={2023}, 
    pages={1-81}
}

@book {GMS,
    AUTHOR = {Garibaldi, Skip and Merkurjev, Alexander and Serre,
              Jean-Pierre},
     TITLE = {Cohomological invariants in {G}alois cohomology},
    SERIES = {University Lecture Series},
    VOLUME = {28},
 PUBLISHER = {American Mathematical Society, Providence, RI},
      YEAR = {2003},
     PAGES = {viii+168},
      ISBN = {0-8218-3287-5},
   MRCLASS = {11E72 (12G05)},
  MRNUMBER = {1999383},
MRREVIEWER = {Gr\'egory\ Berhuy},
       DOI = {10.1090/ulect/028},
       URL = {https://doi.org/10.1090/ulect/028},
}

@book {Alper-book,
    AUTHOR = {Alper, Jared},
     TITLE = {Stacks and Moduli},
    SERIES = {},
    VOLUME = {},
 PUBLISHER = {https://sites.math.washington.edu/~jarod/moduli.pdf},
      YEAR = {2026},
     PAGES = {i-679},
}

@article{kannan-song, 
   title={Graph enumeration for moduli spaces of curves and maps},
   journal={arXiv:2509.18298},
   author={S. Kannan and J. Song},
   year={2025}, 
    pages={1-58}
}

@article {getzler-kapranov,
    AUTHOR = {Getzler, E. and Kapranov, M. M.},
     TITLE = {Modular operads},
   JOURNAL = {Compositio Math.},
  FJOURNAL = {Compositio Mathematica},
    VOLUME = {110},
      YEAR = {1998},
    NUMBER = {1},
     PAGES = {65--126},
      ISSN = {0010-437X,1570-5846},
   MRCLASS = {18C15 (08A02 14H10 57M50)},
  MRNUMBER = {1601666},
MRREVIEWER = {Alexandre\ I.\ Kabanov},
       DOI = {10.1023/A:1000245600345},
       URL = {https://doi.org/10.1023/A:1000245600345},
}

@article {harer-zagier,
    AUTHOR = {Harer, J. and Zagier, D.},
     TITLE = {The {E}uler characteristic of the moduli space of curves},
   JOURNAL = {Invent. Math.},
  FJOURNAL = {Inventiones Mathematicae},
    VOLUME = {85},
      YEAR = {1986},
    NUMBER = {3},
     PAGES = {457--485},
      ISSN = {0020-9910,1432-1297},
   MRCLASS = {32G15 (14H15 57R20)},
  MRNUMBER = {848681},
MRREVIEWER = {William\ Abikoff},
       DOI = {10.1007/BF01390325},
       URL = {https://doi.org/10.1007/BF01390325},
}

@article {chan-galatius-payne,
    AUTHOR = {Chan, Melody and Galatius, S\o ren and Payne, Sam},
     TITLE = {Tropical curves, graph complexes, and top weight cohomology of
              {$\mathcal{M}_g$}},
   JOURNAL = {J. Amer. Math. Soc.},
  FJOURNAL = {Journal of the American Mathematical Society},
    VOLUME = {34},
      YEAR = {2021},
    NUMBER = {2},
     PAGES = {565--594},
      ISSN = {0894-0347,1088-6834},
   MRCLASS = {14T20 (14H10)},
  MRNUMBER = {4280867},
MRREVIEWER = {Paul\ A.\ Hacking},
       DOI = {10.1090/jams/965},
       URL = {https://doi.org/10.1090/jams/965},
}

@article {CFGP,
    AUTHOR = {Chan, Melody and Faber, Carel and Galatius, S\o ren and Payne,
              Sam},
     TITLE = {The {$S_n$}-equivariant top weight {E}uler characteristic of
              {$\mathcal{M}_{g,n}$}},
   JOURNAL = {Amer. J. Math.},
  FJOURNAL = {American Journal of Mathematics},
    VOLUME = {145},
      YEAR = {2023},
    NUMBER = {5},
     PAGES = {1549--1585},
      ISSN = {0002-9327,1080-6377},
   MRCLASS = {14H10 (05A15 18A25 18B40 55R40)},
  MRNUMBER = {4647653},
MRREVIEWER = {Lidia\ Stoppino},
       DOI = {10.1353/ajm.2023.a907705},
       URL = {https://doi.org/10.1353/ajm.2023.a907705},
}

@article{BRW, 
   title={Welschinger--Witt invariants},
   journal={arXiv:2509.04172},
   author={E. Brugall{\'e} and J. Rau and K. Wickelgren},
   year={2025}, 
    pages={1-49}
}

@article{bb-history, 
   title={The Evolution of Enumerative Geometry: A Narrative from Classical Problems to Enriched Invariants},
   journal={arXiv:2510.04275},
   author={C. Bethea and T. Brazelton},
   year={2025}, 
    pages={1-24}
}

@article{bethea-ravi, 
   title={Local multiplicities for an equivariantly enriched non-transverse B{\'e}zout's theorem},
   journal={arXiv:2604.00289},
   author={C. Bethea and C. Ravi},
   year={2026}, 
    pages={1-31}
}

@article {levine-raskit,
    AUTHOR = {Levine, Marc and Raksit, Arpon},
     TITLE = {Motivic {G}auss-{B}onnet formulas},
   JOURNAL = {Algebra Number Theory},
  FJOURNAL = {Algebra \& Number Theory},
    VOLUME = {14},
      YEAR = {2020},
    NUMBER = {7},
     PAGES = {1801--1851},
      ISSN = {1937-0652,1944-7833},
   MRCLASS = {14F42 (55N20 55N35)},
  MRNUMBER = {4150251},
MRREVIEWER = {Matthias\ Wendt},
       DOI = {10.2140/ant.2020.14.1801},
       URL = {https://doi.org/10.2140/ant.2020.14.1801},
}

@article {ayoub-1,
    AUTHOR = {Ayoub, Joseph},
     TITLE = {Les six op\'erations de {G}rothendieck et le formalisme des
              cycles \'evanescents dans le monde motivique. {I}},
   JOURNAL = {Ast\'erisque},
  FJOURNAL = {Ast\'erisque},
    NUMBER = {314},
      YEAR = {2007},
     PAGES = {x+466},
      ISSN = {0303-1179,2492-5926},
      ISBN = {978-2-85629-244-0},
   MRCLASS = {14F20 (14C25 14F42 18A40 18F10 18F20 18G55 19E15)},
  MRNUMBER = {2423375},
MRREVIEWER = {Christian\ Haesemeyer},
}

@book{hartshorne_residues_1966,
    address = {Berlin, Heidelberg},
    series = {Lecture {Notes} in {Mathematics}},
    title = {Residues and {Duality}},
    volume = {20},
    copyright = {http://www.springer.com/tdm},
    isbn = {978-3-540-03603-6 978-3-540-34794-1},
    url = {http://link.springer.com/10.1007/BFb0080482},
    doi = {10.1007/BFb0080482},
    language = {en},
    urldate = {2025-04-03},
    publisher = {Springer Berlin Heidelberg},
    author = {Hartshorne, Robin},
    year = {1966},
}

@article {lee-quantum-invariants,
    AUTHOR = {Lee, Y.-P.},
     TITLE = {Quantum {$K$}-theory. {I}. {F}oundations},
   JOURNAL = {Duke Math. J.},
  FJOURNAL = {Duke Mathematical Journal},
    VOLUME = {121},
      YEAR = {2004},
    NUMBER = {3},
     PAGES = {389--424},
      ISSN = {0012-7094,1547-7398},
   MRCLASS = {14N35 (19E08 53D45 55N15)},
  MRNUMBER = {2040281},
MRREVIEWER = {Andrew\ Kresch},
       DOI = {10.1215/S0012-7094-04-12131-1},
       URL = {https://doi.org/10.1215/S0012-7094-04-12131-1},
}

@article {alper-slicing,
    AUTHOR = {Alper, Jarod},
     TITLE = {Computing invariants via slicing groupoids: {G}el'fand
              {M}ac{P}herson, {G}ale and positive characteristic stable
              maps},
   JOURNAL = {Math. Nachr.},
  FJOURNAL = {Mathematische Nachrichten},
    VOLUME = {285},
      YEAR = {2012},
    NUMBER = {5-6},
     PAGES = {562--579},
      ISSN = {0025-584X,1522-2616},
   MRCLASS = {14L24 (13A50 14D23 14L30 20L05)},
  MRNUMBER = {2902833},
MRREVIEWER = {Dmitry\ A.\ Timash\"ev},
       DOI = {10.1002/mana.201000058},
       URL = {https://doi.org/10.1002/mana.201000058},
}

\end{document}